\documentclass[12pt,reqno]{article}
mm \textheight=8.9in

\usepackage{amssymb}
\usepackage{amsmath}
\usepackage{amsthm}
\usepackage{color}

\newcommand{\br}[3]{{$#1$}$\lower4pt\hbox{$\tp\atop\raise4pt \hbox{$\scriptscriptstyle{#2}$}$} ${$#3$}}
\newcommand{\tw}[3]{{$#1$}${\,\scriptscriptstyle {#2}}\atop\raise9pt\hbox{$\scriptstyle\tp$} ${$#3$}}
\newcommand{\ttps}[2]{{#1}\raise5pt\hbox{$\lower12pt\hbox{$\scriptstyle\tp$}\atop \lower0pt\hbox{$\tilde\;$}$}\raise4.5pt\hbox{${\scriptstyle{#2}}$}}
\newcommand{\st}[1]{\mbox{${\,\scriptscriptstyle {#1}}\atop\raise5.5pt\hbox{$*$}$}}

\newcommand{\rd}[1]{\mbox{${\,\scriptscriptstyle {#1}}\atop\raise5.5pt\hbox{$\bullet$}$}}
\newcommand{\rt}[1]{\otimes_\chi}
\newcommand{\lt}[1]{\mbox{${\,\scriptscriptstyle {#1}}\atop\raise5.5pt\hbox{$\ltimes$}$}}
\newcommand{\btr}{\raise1.2pt\hbox{$\scriptstyle\blacktriangleright$}\hspace{2pt}}
\newcommand{\btl}{\raise1.2pt\hbox{$\scriptstyle\blacktriangleleft$}\hspace{2pt}}

\newcommand{\lcr}{\raise1.0pt \hbox{${\scriptstyle\rightharpoonup}$}}
\newcommand{\rcr}{\raise1.0pt \hbox{${\scriptstyle\leftharpoonup}$}}

\newcommand{\ttp}{{\lower12pt\hbox{$\tp$}\atop \hbox{$\tilde\;$}}}

\newcommand{\Hg}{\mathfrak{H}}
\newcommand{\Pg}{\mathfrak{P}}

\newcommand{\Bc}{\mathcal{B}}

\newcommand{\A}{\mathcal{A}}

\newcommand{\C}{\mathbb{C}}

\newcommand{\Z}{\mathbb{Z}}

\newcommand{\Sbb}{\mathbb{S}}

\newcommand{\tp}{\otimes}

\newcommand{\V}{V}
\newcommand{\U}{U}

\newcommand{\Fc}{\mathcal{F}}
\newcommand{\ve}{\varepsilon}
\newcommand{\gm}{\gamma}
\newcommand{\dt}{\delta}

\newcommand{\op}{\oplus}
\newcommand{\la}{\lambda}
\newcommand{\tr}{\triangleright}
\newcommand{\tl}{\triangleleft}

\newcommand{\End}{\mathrm{End}}

\newcommand{\Span}{\mathrm{Span}}

\newcommand{\Hom}{\mathrm{Hom}}

\newcommand{\Rm}{\mathrm{R}}

\newcommand{\ad}{\mathrm{ad}}

\newcommand{\La}{\Lambda}

\newcommand{\g}{\mathfrak{g}}
\renewcommand{\b}{\mathfrak{b}}
\renewcommand{\k}{\mathfrak{k}}
\newcommand{\h}{\mathfrak{h}}

\newcommand{\s}{\mathfrak{s}}
\renewcommand{\o}{\mathfrak{o}}

\newcommand{\nn}{\nonumber}

\newcommand{\p}{\mathfrak{p}}
\renewcommand{\l}{\mathfrak{l}}

\newcommand{\al}{\alpha}

\newcommand{\bt}{\beta}

\newcommand{\be}{\begin{eqnarray}}
\newcommand{\ee}{\end{eqnarray}}

\newtheorem{thm}{Theorem}[section]
\newtheorem{propn}[thm]{Proposition}
\newtheorem{lemma}[thm]{Lemma}
\newtheorem{corollary}[thm]{Corollary}

\theoremstyle{definition}

\newcount\prg

\newcommand{\parag}{\advance\prg by1 {\noindent\bf\thesection.\the\prg\hspace{6pt}}}

\begin{document}
\title{Star-product on complex sphere $\mathbb{S}^{2n}$}
\author{
A. Mudrov\footnote{This study is supported in part by the RFBR grant 15-01-03148.}\vspace{20pt}\\
{\em \small  Dedicated to the memory of Petr Kulish}
\vspace{10pt}\\
\small Department of Mathematics,\\ \small University of Leicester, \\
\small University Road,
LE1 7RH Leicester, UK\\
}

\date{ }

\maketitle
\begin{abstract}
We construct a  $U_q\bigl(\s\o(2n+1)\bigr)$-equivariant local star-product on the complex
sphere $\mathbb{S}^{2n}$ as a non-Levi conjugacy class $SO(2n+1)/SO(2n)$.
\end{abstract}
\\
{\small \underline{AMS classification codes}: 17B10, 17B37, 53D55.}

\section{Introduction}
In this paper, we incorporate a simple example of homogeneous space with  non-Levi stabilizer into a uniform quantization
scheme for closed conjugacy classes of simple algebraic groups. This approach was developed in 2003 for Levi classes
and utilized the presence of quantum isotropy subgroup in the total quantum group, \cite{DM,EEM}.
 The key distinction of non-Levi classes
is the absence of a natural candidate for such a subgroup because its root basis cannot be made
a part of the total root basis. Still the coordinate ring of the class can be quantized by an operator
realization on certain modules, \cite{AM}. Such a quantization is formulated in terms of generators and relations and is not apparently local.
On the other hand,  a dynamical twist constructed from the Shapovalov
form yields a local version of the star product on Levi classes, \cite{DM,EEM} (see also \cite{AL,EE} for coadjoint orbits with the Kirillov
bracket). It is  natural to extend that approach to all closed conjugacy classes. Such a possibility for $\Sbb^4$ was pointed out
without proof in \cite{M1}. Here we give a solution for all even dimensional spheres.

The original approach to the star product on Levi classes was as follows. Let $\k\subset \g$ be the isotropy Levi subalgebra
of a point $t$ and $\p_\pm\subset\g$ its
parabolic extensions. The point $t$ is associated with a certain weight $\la\in \h^*$ and a pair of modules
$M_\la$, $N_\la$ of, respectively, highest and lowest weights $\la$ and $-\la$. There is a unique $U_q(\g)$-invariant  form
$M_\la\tp N_\la\to \C$, which is non-degenerate if and only if the modules are irreducible. In that case, there
 exists the inverse form $\C\to N_\la\tp M_\la$ and its lift $1\mapsto  \Fc\in U_q(\p_+)\tp U_q(\p_-)$.
The element $\Fc$ gives rise to a "bidifferential" operator via the left co-regular action on the Hopf dual $\A=U_q^*(\g)$.
With this operator, the multiplication in  $\A$ is twisted to a non-associative operation  invariant under the right
co-regular action of $U_q(\g)$. The key observation is  that the new multiplication becomes associative when restricted to the subspace
$\A^\k$ of $U_q(\k)$-invariants in $\A$. As a (right) $U_q(\g)$-module, $\A$ has the same structure as  the $U(\g)$-module $\C[G]$, hence
$\A^\k$ is a flat deformation of $\C[G/K]$. It is known that the initial star product on $\A$ is local, \cite{EK}, therefore the resulting
multiplication is local as well.

In the non-Levi case, one can go along those lines and define $\A^\k$ as the joint kernel of certain operators that deform
generators of $\k$. Then the new product will be associative on $\A^\k$ as in the Levi case, \cite{KST}. However, those operators do not close up to
a deformation of $U_q(\k)$ so one cannot  be sure that $\A^\k$ has the proper size (observe that kernel can decrease under deformation).
Therefore the problem is to check the size of $\A^\k$. We do it for even dimensional spheres regarded as conjugacy classes of $SO(2n+1)$.
Note that odd dimensional spheres  belong to the second connected component of the orthogonal group $O(2n)$,  and the current methods are not directly applicable.

The paper consists of five sections. After the introduction we recall  quantization of $\C[\Sbb^{2n}]$ via operator realization
on a highest weight module $M_\la$ in Section \ref{QG}. In the next section we construct a system of vectors that spans $M_\la$.
We prove it to be a basis in Section \ref{Shap} by computing the Shapovalov form on $M_\la$. This way
we show that $M_\la$ is irreducible and the form is invertible. In the final section  we show that for locally  finite
$U_q(\g)$-module $V_q$, the dimension of $V^\k_q$ is equal to $\dim V^\k$ of the classical $\k$-invariants. We do it
via realization of finite dimensional module  $V_q$ with $\dim V_q^\k>0$ in the coordinate ring of the quantum Euclidean plane $\C^{2n+1}_q$.
This way we complete the task.

\section{Operator realization of $\C_q[\Sbb^{2n}]$}
\label{QG}
Throughout the paper, $\g$ stands for the Lie algebra $\s\p(2n+1)$.
We are looking for quantization of the polynomial ring $\C[\mathbb{S}^{2n}]$
that is invariant
under an action of the quantized universal enveloping algebra $U_q(\g)$.
We regard $\Sbb^{2n}$ as a conjugacy class of the Poisson group $G=SO(2n+1)$ equipped with the Drinfeld-Sklyanin
bracket corresponding to the standard solution $r\in \g\tp \g$ of the classical Yang-baxter equation, \cite{D}.
The group $G$ supports the Semenov-Tian-Shansky bivector field
\be
r_-^{l,l}+r_-^{r,r}-r_-^{r,l}-r_-^{l,r}+r_+^{r,l}-r_+^{l,r}
\label{STSbr}
\ee
making it a Poisson $G$-space with respect to conjugation.
Here $r_-$ and $r_+$ are, respectively, the skew-symmetric and invariant symmetric parts of $r$, and the superscripts designate
the vector fields
$$
(\xi^l f)(g)=\frac{d}{dt}f(ge^{t\xi})|_{t=0},
\quad
(\xi^r f)(g)=\frac{d}{dt}f(e^{t\xi}g)|_{t=0},
$$
where $\xi\in \g$ and $f$ is a smooth function $G$. This bivecftor field (\ref{STSbr}) is tangent to every conjugacy class of $G$. In particular,
the sphere $\Sbb^{2n}$ becomes a homogeneous Poisson-Lie manifold over $G$.

Quantization of $\C[G]$ along  (\ref{STSbr}) gives rise to the reflection equation dual
of $U_q(\g)$, \cite{DM1}. Accordingly, the algebra $\C[\Sbb^{2n}]$ can be presented as its quotient.
Here we recall that construction.

Let $\h\subset \g$ denote the Cartan subalgebra equipped with the inner product  restricted from an $\ad$-invariant form on $\g$.
We endow the dual space $\h^*$  with the inverse form, $(.,.)$. For
any $\mu\in \h^*$ we denote by $h_\mu\in \h$ the vector such that $\nu(h_\mu)=(\nu,\mu)$ for all $\nu\in \h^*$.
We normalize the inner product   so that the short roots have length $1$.

The root system $\Rm$
is expressed through the orthogonal basis $\La=\{\ve_i\}_{i=1}^n\subset \h^*$
consisting of  short roots: the basis of simple positive roots  $\Pi$ is formed by  $\al_1=\ve_1, \al_{i}=\ve_{i}-\ve_{i-1}$, $i=2,\ldots, n$.
We distinguish the subalgebra $\l\simeq \g\l(n)\subset \g$ of maximal rank with the root basis $\Pi_\l=\{\al_i\}_{i=2}^n$.

Throughout the paper we
assume that $q\in \C$ is not a root of unity.
We use the notation $\bar q=q^{-1}$, $[z]_q=\frac{q^z-q^{-z}}{q-q^{-1}}$, and  $[x,y]_a=xy-ayx$ for $a\in \C$.
The quantum group $U_q(\g)$ is a  $\C$-algebra generated by $q^{\pm h_\al}$, $e_{\pm \al}$, $\al\in \Pi$, such that
$
q^{h_{\al}}e_{\pm \bt}q^{- h_{\al}}= q^{\pm(\al,\bt)} e_{\pm\bt}
$
 and $[e_{\al},e_{-\bt}]=\dt_{\al,\bt}[h_\al]_q$ forall  $\al,\bt \in \Pi$.
The generators $e_{\pm\al}$ satisfy the  q-Serre relations
$$
[e_{\pm\al},[e_{\pm\al},e_{\pm\bt}]_q]_{\bar q}=0, \quad \forall \al,\bt\in \Pi\quad\mbox{s.t.}\quad
 \frac{2(\al,\bt)}{(\al,\al)}=-1, \quad\mbox{and }\quad [e_{\pm\al_1},e_{\pm\dt}]=0,
$$
where $e_{\pm\dt}=[e_{\pm \al_1},[e_{\pm \al_1},e_{\pm \al_2}]_q]_{\bar q}$.
Also, $[e_{\pm \al},e_{\pm \bt}]=0$ once $(\al,\bt)=0$, \cite{D}.

The subset $\Pi_\k=\{\dt,\al_1,\ldots,\al_n\}\subset \Rm^+$ forms
a root basis for  a subalgebra $\k\subset \g$ isomorphic to $\s\o(2n)$. Although
$e_{\pm\dt}$ are deformations of classical root vectors, they do not generate an  $\s\l(2)$-subalgebra in $U_q(\g)$,
so we have no natural subalgebra $U_q(\k)$ in $U_q(\g)$. Still $e_{\pm\dt}$ play a role in what follows.

Fix the weight $\la\in \h^*$ by the conditions $q^{2(\la,\ve_i)}=-q^{-1}$ for all $i=1,\ldots,n$, and $(\al_i, \la)=0$ for $i>1$.
Define two one-dimensional representations $\C_{\pm\la}$  of $U_q(\l)$  by $e^{h_\al}\mapsto q^{\pm (\la,\al)}$, $\al\in \Pi_\g$
and by zero on the generators on non-zero weight. Extend them to representations of $U_q(\p_\pm)$ by zero
on $e_{\pm \al}$ for all  $\al \in \Pi_\g$.
Then set
$$
\hat M_\la=U_q(\g)\tp_{U_q(\p_+)} \C_\la, \quad \hat N_\la=U_q(\g)\tp_{U_q(\p_-)} \C_{-\la}.
$$
Denote by $1_\la\in M_\la$ and $1^*_\la \in N_\la$ their highes/lowest weight generators.
Due to the special choice of $\la$,
the vectors
$
e_{-\dt}1_\la\in \hat M_\la$ and
$e_{\dt}1^*_\la\in \hat N_\la
$
are killed by $e_\al$ and, respectively, $e_{-\al}$ for all $\al\in \Pi$. They generate submodules $\hat M_{\la-\dt}\subset M_\la$ and $\hat N_{\la-\dt}\subset N_\la$.
Set $M_{\la}=\hat M_\la/\hat M_{\la-\dt}$ and $N_\la=\hat N_{\la}/\hat N_{\la-\dt}$.

The module $M_\la$ supports quantization of $\C[\Sbb^{2n}]$ in the following sense.
The sphere $\Sbb^{2n}$ is isomorphic  a subvariety in $G$  of orthogonal matrices with eigenvalues $\pm 1$, where $1$ is multiplicity free.
It is a conjugacy class with a unique point of intersection with the maximal torus relative to $\h$.
The  isotropy subalgebra of this point is $\k$.
Quantization of $\C[G]$ along the Poisson bracket (\ref{STSbr}) can  be  realized as a subalgebra   $\C_q[G]\subset U_q(\g)$ invariant under the adjoint action.
The image of $\C_q[G]$ in $\End(M_\la)$ is an equviariant quantization of $\C[\Sbb^{2n}]$, see \cite{M2} for details.

\section{Spanning $M_\la$}
In this section we introduce a set  of vectors in $M_\la$ which is  proved to be a basis in the subsequent section.
Here we prove that it spans $M_\la$.
Put $f_{\al}=e_{-\al}$ for all simple roots  and define
$$
f_{\ve_i}=[\ldots [f_{\al_1},f_{\al_2}]_{\bar q},\ldots f_{\al_i}]_{\bar q}\in U_q(\g_-),
\quad 1\leqslant i\leqslant n.
$$
The elements $f_{\ve_i}$ can be included in  the set of composite root vectors generating a Poincare-Birkhoff-Witt basis in $U_q(\g_-)$, \cite{ChP}. By deformation arguments,
the set of monomials $\Bc =\{f_{\ve_1}^{m_{\ve_1}}\ldots f_{\ve_n}^{m_{\ve_n}}1_\la\}_{m_1,\ldots,m_n\in \Z_+}$ is a basis in $M_\la$ extended over the local ring
$\C[[\hbar]]$, where $\hbar =\log q$, see \cite{M2} and references therein.
We prove that $\Bc$ is  a $\C$-basis once  $q$ is not a root of unity.

Let $\k_m^-$ denote the subspace $\C f_\dt+\Span\{f_{\al_2},\ldots,f_{\al_m}\}\subset U_q(\g_-)$.
\begin{lemma}
\label{normalizer}
  For all $1<i\leqslant m$, the elements $f_{\ve_i}$ belong to the normalizer of the left ideal $ U_q(\g)\k_m^-$.
\end{lemma}
\begin{proof}
Serre relations readily yield  $[f_{\al_m}, f_{\ve_m}]_{\bar q}=0$, for $m>1$.
  For $1<i<m$, the identity $[f_{\al_i}, f_{\ve_m}]=0$ follows from Lemma \ref{lemma_XYZ}.
So we are left to study how $f_{\al_i}$ commute with $f_\dt$.


  The element $f_\dt$ commutes with  $f_{\ve_2}$, which completes the proof for $m=2$. Suppose that  $m=3$.
All calculations below are done modulo $U_q(\g)\k_m^-$. Denote $a=[2]_q$, then
\be
\label{mod-rel}
f_{\al_1}f_{\al_2}f_{\al_1} =\frac{1}{a}f_{\al_2}f_{\al_1}^{2}
,\quad
f_{\al_2}f_{\al_1}^3=(a-\frac{1}{a})f_{\al_1}f_{\al_2}f_{\al_1}^2
,\quad
f_{\al_2}f_{\al_1}^3=(a^2-1)f_{\al_1}^2f_{\al_2}f_{\al_1},
\ee
where the left equality means $f_\dt \in U_q(\g)\k_3^-$, and the last two equalities are obtained from it and
from $[f_\dt,f_{\al_1}]=0$. Furthermore, Serre relations along with (\ref{mod-rel}) yield
\be
 f_{\al_1}f_{\al_2}(f_{\al_1}  f_{\al_3})f_{\al_2}f_{\al_1}&=& f_{\al_1}f_{\al_2} f_{\al_3}(f_{\al_1} f_{\al_2}f_{\al_1})= \frac{1}{a} f_{\al_1}(f_{\al_2} f_{\al_3}f_{\al_2})f_{\al_1}^2
= \frac{1}{a^2} f_{\al_3}(f_{\al_1}f_{\al_2}^2)f_{\al_1}^2
\nn
\\
&=&
\frac{1}{a}f_{\al_3}f_{\al_2}(f_{\al_1}f_{\al_2}f_{\al_1}^2)-\frac{1}{a^2} f_{\al_3}f_{\al_2}^2f_{\al_1}^3=\frac{1}{a^2-1}f_{\al_3}f_{\al_2}^2f_{\al_1}^3-\frac{1}{a^2} f_{\al_3}f_{\al_2}^2f_{\al_1}^3
\nn
\\
&=&
\frac{1}{a^2(a^2-1)}f_{\al_3}f_{\al_2}^2f_{\al_1}^3,
\nn\\
f_{\al_2}(f_{\al_1}^2  f_{\al_3})f_{\al_2}f_{\al_1}
&=&
f_{\al_2} f_{\al_3} (f_{\al_1}^2 f_{\al_2}f_{\al_1})
=\frac{1}{(a^2-1)}f_{\al_2} f_{\al_3} f_{\al_2}f_{\al_1}^3.
\nn
\ee
Multiply the first equality by $a$ and subtract from the second:
$$
f_\dt f_{\ve_3}=(f_{\al_2}f_{\al_1}^2 - af_{\al_1}f_{\al_2}f_{\al_1}) f_{\al_3}f_{\al_2}f_{\al_1}=
\frac{1}{a(a^2-1)}(af_{\al_2} f_{\al_3} f_{\al_2}-f_{\al_3}f_{\al_2}^2)f_{\al_1}^3 \in U_q(\g)\k_3^-.
$$
This completes the case $m=3$. For $m\geqslant i>3$, put $f_{\ve_i-\ve_3}=[f_{\al_4},\ldots[f_{\al_{i-1}},f_{\al_i}]_{\bar q}\ldots]_{\bar q}\in U_q(\g)\k_m^-$. Then
$
f_\dt  f_{\ve_i}=-q^{-1}f_\dt f_{\ve_i-\ve_3}f_{\ve_3}=-q^{-1}f_{\ve_i-\ve_3}f_\dt f_{\ve_3}\in U_q(\g)\k_m^-,
$
as required.
\end{proof}
\begin{corollary}
\label{span}
  The set $\Bc$ spans $M_\la$.
  The action of $U_q(\g_-)$ on $M_\la$ is given by
  \be
f_{\al_1} f_{\ve_1}^{m_1}\ldots f_{\ve_n}^{m_n}1_\la&=&f_{\ve_1}^{m_1+1} f_{\ve_{2}}^{m_{2}}
\ldots f_{\ve_n}^{m_n}1_\la,
\nn\\
f_{\al_{i+1}} f_{\ve_1}^{m_1}\ldots f_{\ve_n}^{m_n}1_\la&=&-q[m_{i}]_qf_{\ve_1}^{m_1}\ldots f_{\ve_{i}}^{m_{i}-1} f_{\ve_{i+1}}^{m_{i+1}+1}
\ldots f_{\ve_n}^{m_n}1_\la, \quad i>1.
\nn
\ee
\end{corollary}
\begin{proof}
First let us show that $f_{\ve_{i+1}}f_{\ve_{i}}= q^{-1}f_{\ve_{i}}f_{\ve_{i+1}}\mod U_q(\g)\k^-_{i+1}$. Indeed,
for $i=1$ we
have $f_{\ve_2}f_{\ve_1}= q^{-1} f_{\ve_1}f_{\ve_2}- q^{-1} f_\dt=q^{-1} f_{\ve_1}f_{\ve_2}\mod U_q(\g)\k_2^-$ as required. For $i>2$ we get
$$
[f_{\ve_i},f_{\ve_{i+1}}]_q=[[f_{\ve_1},f_{\ve_i-\ve_2}]_{\bar q},f_{\ve_{i+1}}]_q
=[f_{\ve_1},[f_{\ve_i-\ve_2},f_{\ve_{i+1}}]]+[[f_{\ve_1},f_{\ve_{i+1}}]_q,f_{\ve_i-\ve_2}]_{\bar q}.
$$
The first summand vanishes since $f_{\ve_i-\ve_2}$ commutes with $f_{\ve_{i+1}}$, by Lemma \ref{lemma_XYZ}. The internal commutator
in the second summand is $[f_\dt,f_{\ve_{i+1}-\ve_2}]\in U_q(\g)\k^-_{i+1}$, so this term is in $U_q(\g)\k^-_{i+1}$ as well.

Now we can complete the proof.  The linear span $\C\Bc$ is
invariant under the obvious action of $f_{\al_1}=f_{\ve_1}$.
For $i>0$, we push $f_{\al_{i+1}}$ to the right in the product
$$
f_{\al_{i+1}}f_{\ve_1}^{m_1}\ldots  f_{\ve_n}^{m_n}1_\la=f_{\ve_1}^{m_1}\ldots f_{\al_{i+1}}f_{\ve_{i}}^{m_i}\ldots  f_{\ve_n}^{m_n}1_\la.
$$
Thanks to Lemma \ref{normalizer}, we replace the product $f_{\al_{i+1}}f_{\ve_{i}}^{m_i}$ with
$$
[f_{\al_{i+1}},f_{\ve_{i}}^{m_i}]_{q^{m_i}}=
-q\sum_{l=0}^{m_i-1-l}q^lf_{\ve_{i}}^lf_{\ve_{i+1}}f_{\ve_{i}}^{m_i-1-l}=-q[m_i]_q f_{\ve_{i}}^{m_i-1}f_{\ve_{i+1}}
\mod U_q(\g)\k^-_{j+1}.
$$
For any form $\Phi$ in $n-i$ variables, the ideal $U_q(\g)\k^-_{j+1}$ kills  $\Phi(f_{\ve_{i+1}},\ldots,  f_{\ve_n})1_\la$
by Lemma \ref{normalizer}.
This yields the action of $f_{\al_{i+1}}$ on $\C\Bc$ and proves its $U_q(\g_-)$-invariance.
Since $\C\Bc\ni 1_\la$,  it coincides with $M_\la$.
\end{proof}

\section{Invariant bilinear form $M_\la\tp N_\la\to \C$}
\label{Shap}
Introduce  positive root vectors by
$$
e_{\ve_i}=[\ldots[e_{\al_{i}},e_{\al_{i-1}}]_q,\ldots e_{\al_1}]_{q},
\quad 1\leqslant i\leqslant n.
$$
The elements $f_{\ve_i},e_{\ve_i}$ are known to generate $U_q(\s\l(2))$-subalgebras in $U_q(\g)$ with
the commutation relation $[e_{\ve_i},f_{\ve_i}]=\frac{q^{h_{\ve_i}}-q^{-h_{\ve_i}}}{q-q^{-1}}$.
Define by induction  $\tilde e_{\ve_{i+1}}=[e_{\al_{i+1}},\tilde e_{\ve_{i}}]_{\bar q}$ with $ \tilde e_{\ve_1}=e_{\al_1}$.
Then $\omega(f_{\ve_i})=\tilde e_{\ve_i}$, where $\omega$ is the Chevalley anti-algebra involution of $U_q(\g)$.

Fix the comultiplication on $U_q(\g)$ as in \cite{ChP}:
$$
\Delta(e_{\al})=e_{\al}\tp q^{h_\al}+1\tp e_{\al}, \quad \Delta(f_{\al})=f_{\al}\tp 1+q^{-h_\al}\tp f_{\al},
$$
and $\Delta(q^{h_\al})=q^{h_\al}\tp q^{h_\al}$, for all $\al \in \Pi$. Then $\gm^{-1}(e_{\al})=-q^{-h_{\al}}e_\al$
for the inverse antipode $\gm^{-1}$.
Define a map $U_q(\g)\to\C$, $x\mapsto \langle x\rangle$, as the composition of
the projection $U_q(\g)\to U_q(\h)\mod \g_-U_q(\g)+U_q(\g)\g_+$ with evaluation at $\la$.
The assignment $(x 1_\la,y1^* _\la)= \langle \gm^{-1}(y)x\rangle$, $\forall x,y\in U_q(\g)$,
defines  a unique invariant bilinear form $M_\la\tp N_\la\to \C$ such that $(1_\la,1^*_\la)=1$.
\begin{lemma}
Suppose that  $k_i,m_i \in \Z_+$, for $i=1,\ldots,n$. Then
\be
  \langle e_{\ve_n}^{k_n}\ldots e_{\ve_1}^{k_1} f_{\ve_1}^{m_1}\ldots  f_{\ve_n}^{m_n}\rangle
  =\prod_{i=1}^{n}\dt_{k_i,m_i} [m_i]_q!\theta^{-m_i}q^{-m_i(\la,\ve_i)-\frac{m_i}{2}}(-1)^{m_i},
  \label{factorization}
\ee
where $\theta=q^{\frac{1}{2}}-q^{-\frac{1}{2}}$.
\end{lemma}
\begin{proof}
  The $\dt$-symbols  are due to orthogonality of weight subspaces $M_\la[\mu]$ and $N_\la[\nu]$ unless $\mu=-\nu$.
  Now we prove factorization of the matrix coefficients on setting $k_i=m_i$ for all $i$.
  Observe that
\be
e_{\ve}^{m} f_{\ve}^{m}
=
\prod_{l=1}^m\frac{[\frac{l}{2}]_q}{[\frac{1}{2}]_q}\prod_{l=0}^{m-1}[h_{\ve}-\frac{l}{2}]_q \mod U_q(\g)e_{\ve}, \quad \forall\ve\in \La,
\label{Harish}
\ee
and $\langle e_{\ve}^{m} f_{\ve}^{m}\rangle=(-1)^mq^{-m(\la,\ve)-\frac{m}{2}}\theta^{-m}[m]_q!$ on  substitution $q^{2(\la,\ve)}=-q^{-1}$.
 Suppose we have proved that the l.h.s. of (\ref{factorization}) is equal to
$
  \langle e_{\ve_n}^{m_n}\ldots e_{\ve_s}^{m_s} f_{\ve_s}^{m_s}\ldots  f_{\ve_n}^{m_n}\rangle\prod_{l=1}^{s-1}\langle e_{\ve_{l}}^{m_{l}} f_{\ve_{l}}^{m_{l}}\rangle
$
for some $s=1,\ldots, n-1$. For all $i>s$, $[e_{\ve_s},f_{\ve_i}]\in U_q(\g)\k^-_{i}$. Then
$e_{\ve_s}\ldots  f_{\ve_{s+1}}^{m_{s+1}}\ldots  f_{\ve_n}^{m_n}1_\la=0$, by Lemma \ref{normalizer}. Now the presentation
(\ref{Harish}) for $\ve=\ve_s$, along with orthogonality of different $\ve_i$, gives
$
  \langle e_{\ve_n}^{m_n}\ldots e_{\ve_{s-1}}^{m_{s-1}} f_{\ve_{s-1}}^{m_{s-1}}\ldots  f_{\ve_n}^{m_n}\rangle\prod_{l=1}^{s}\langle e_{\ve_{l}}^{m_{l}} f_{\ve_{l}}^{m_{l}}\rangle
$. Induction on $s$ completes the proof.
\end{proof}
There is also $\omega$-contravariant form on $M_\la$ defined by
$x1_\la\tp y1_\la\mapsto \langle \omega(x)y\rangle$, for all $x,y\in U_q(\g)$.
It is called the Shapovalov form and related with the invariant form in the obvious way.
\begin{propn}
Suppose that $q$ is not a root of unity. Then
\begin{enumerate}
  \item $\Bc\subset M_\la$ is an orthogonal (non-normalized)  basis with respect to the Shapovalov form.
  \item The modules $M_\la$ and $N_\la$ are irreducible.
  \item
The tensor
$$
 \Fc=\sum_{m_1,\ldots,m_n=0}^{\infty}(-\theta)^{\sum_{i=1}^{n}m_i}\>
\frac{\prod_{i=1}^{n}q^{-\frac{m_i^2}{2}+2m_i(i-1)}}{\prod_{i=1}^n[m_i]_{q}!}
\tilde e_{\ve_1}^{m_1}\ldots \tilde e_{\ve_n}^{m_n} \tp f_{\ve_1}^{m_1}\ldots  f_{\ve_n}^{m_n}
$$
is a lift of the inverse invariant form, $\C\to N_\la \tp M_\la\to  U_q(\g_+)\tp U_q(\g_-)$, $1\mapsto  \Fc$.
\end{enumerate}
\end{propn}
\begin{proof}
1) Corollary \ref{span} with Lemma \ref{factorization} prove completeness of $\Bc$ and independence.
Since $\langle e_{\ve_n}^{k_n}\ldots e_{\ve_1}^{k_1} f_{\ve_1}^{m_1}\ldots  f_{\ve_n}^{m_n}\rangle=q^{\sum_{i=1}^{n}2k_i(i-1)}\langle \omega(f_{\ve_1}^{k_1}\ldots f_{\ve_n}^{k_n}) f_{\ve_1}^{m_1}\ldots  f_{\ve_n}^{m_n}\rangle$,
 the basis $\Bc$ is orthogonal with respect to the Shapovalov form.
 2) Non-degeneracy of the form implies irreducibility of $M_\la$.
 3) The normalizing coefficients in $ \Fc$ are  obtained from (\ref{factorization}) via the equality  $(x 1_\la,\gm(y)1^*_\la)= \langle  y x \rangle$
  and $\gm(e_{\ve_i})=-\tilde e_{\ve_i}q^{-h_{\ve_i}+2(i-1)}$ for all $i=1,\ldots,n$.
\end{proof}
\section{Star-product on $\Sbb^{2n}$}
Denote by $\A_q$  the RTT dual of $U_q(\g)$ with multiplication $\bullet$ and the Hopf paring
$(.,.)$.
It is equipped with the two-sided action (here $a^{(1)}\tp a^{(2)}=\Delta(a)$ in the Sweedler notation)
$$
x \tr a=a^{(1)}(a^{(2)},x), \quad a\tl x=(a^{(1)},x)a^{(2)}, \quad \forall x\in U_q(\g), \quad a \in \A_q,
$$
making it a $U_q(\g)$-bimodule algebra.
The multiplication $\bullet$ is known to be local, \cite{EK}. We define a new operation $\star$ by
\be
\label{star}
f\star g= (\Fc_1\tr f )\bullet(\Fc_2 \tr g), \quad f,g \in \A_q,
\ee
where $\tr$ stands for the left co-regular action.
It is obviously equivariant with respect to the right co-regular action of $U_q(\g)$.
However, $\star$ is not associative on the entire $\A_q$.

For every $\U_q(\g)$-module $V$ we define $V^\k\subset V$ to be the intersection of the space  $V^\l$ of $U_q(\l)$-invariants  with the joint kernel of the operators $e_\dt$ and $f_\dt$. For $q=1$, this definition coincides with the subspace  of $U(\k)$-invariants.

\begin{propn}
   $\A^\k_q$ is an associative $U_q(\g)$-algebra  with respect to $\star$.
\end{propn}
\begin{proof}
Identify $M_\la^*$ with $N_\la^{**}$ and $M_\la\tp M_\la^*$ with the module of locally finite endomorphisms $\End^\circ_\C(M_\la)$.
   For every completely reducible  module $V$, there is a unique $\tilde \phi\in \Hom(V^*,\End_\C^\circ(V))$ for  each $\phi\in \Hom(M_\la\tp N_\la,\V)$,
   due to the natural isomorphism of the Hom-sets.

  Let $\hat M_\la$ and $\hat N_\la$ denote the Verma modules, i. e. induced from the $U_q(\b_\pm)$-modules $\C_{\mp \la}$.
  Every homomorphism $M_{\la}\tp N_{\la} \to    V$ amounts to a homomorphism $\hat M_{\la}\tp \hat N_{\la} \to  V$
  vanishing on $\sum_{\al \in \Pi_\k}\hat M_{\la-\al}\tp \hat  N_{\la}+\sum_{\al \in \Pi_\k}\hat  M_{\la}\tp  \hat N_{\la-\al}$.
Therefore  $\phi$ corresponds  to a unique zero weight element $\Phi(\phi)\in \cap_{\al \in \pm \Pi_\k}\ker e_\al = V^\k$.
Given also $\psi\in\Hom(M_\la\tp N_\la, W)$ there is a unique element $\phi\circledast \psi \in \Hom(V\tp W,\End_\C^\circ(M))$ such that
 $\widetilde{\phi\circledast \psi}=\tilde \phi\circ \tilde \psi$,  where $\circ$ is the multiplication in
$\End^\circ(M_\la)$. Define  $\Phi(\phi)\circledast  \Phi(\psi)=\Phi(\phi\circledast \psi)\in (V\tp W)^\k$.

 Now take $V=W=\A_q$ and $f,g\in \A_q^\k$ (with respect to the $\tr$-action).
Then  $f\star g$ is the image of $f\circledast  g \in (\A_q\tp \A_q)^\k$ under the multiplication $\bullet\colon \A_q\tp \A_q\to \A_q$,
which is again in $\A_q^\k$ since $\bullet$ is $\tr$-equivariant.
Associativity of $\star$ follows from associativity of $\circ $ and $\bullet$.
\end{proof}
\begin{thm}
\label{main_thm}
  The right $\U(\g)$-module $\A^\k_q$ is a deformation of the $\U(\g)$-module $\C[G]^\k$. The  multiplication $\star$  makes
 $\A^\k_q$ an associative $\U_q(\g)$-algebra, a quantization of $\C[\Sbb^{2n}]$.
\end{thm}
\begin{proof}
  We only need to make sure that $\A^\k_q=\op_V\> {}^* V^\k\tp \!V$ is a deformation of $\C[\Sbb^{2n}]\simeq \C[G]^\k$. It is done
  by Proposition \ref{invariants} below.
\end{proof}
\noindent
Remark, that despite $\A_q^\k$ goes over to $\C[G]^\k$ at $q=1$, the fact $\A_q^\k\simeq \C[G]^\k\tp \C[q,q^{-1}]$ needs a proof
because  $\ker e_\dt$ and $\ker f_\dt$ may decrease under deformation. That is done in the next section.
\section{Quantum Euclidian plane}
To complete the proof of Theorem \ref{main_thm}, it is sufficient to check $\dim V_q^\k=\dim V^\k$ for all finite dimensional modules $V$ that appear in $\C[\Sbb^{2n}]$. They all can be realized
in the polynomial ring of the Euclidian plane $\C^{2n+1}$, \cite{Vil}. So we going to look at its quantum version.

Choose a basis $\{x_i\}_{i=-n}^n\subset \C^{N}$, $N=2n+1$, and define
a representation of $U_q(\g)$ on $\C^{N}$  by the assignment
$$
e_{\al_i}\tr x_k= \dt_{k,i-1}x_i- \dt_{k,-i}x_{-i+1}, \quad f_{\al_i}\tr x_k= \dt_{k,i}x_{i-1}- \dt_{k,-i+1}x_{-i}
$$
for $i=1, \ldots,n$. Then $x_i$ carry weights  $\ve_{i}$ subject to  $\ve_{i}=-\ve_{-i}$.
The quantum Euclidian  plane $\C_q[\C^{N}]$ is an associative algebra generated by $\{x_i\}_{i=-n}^n$ with relations
$$
x_i x_j=q^{-1}x_jx_i, \quad i>j, \quad i\not= j,j'
,\quad x_{1}x_{-1}-x_{-1}x_{1}=(q-1)x_{0}^2,\quad
$$
$$
x_{j}x_{-j}-x_{-j}x_{j}=q x_{j-1}x_{-j+1}-q^{-1}x_{-j+1}x_{j-1}, \quad j>1.
$$
These relations are equivalent to those presented in $\cite{FRT}$.

The representation on $\C^N$ extends to an action $\tr$ on $\C_q[\C^N]$ making it a $\U_q(\g)$-module algebra.
Let $\theta$ denote the involutive algebra and anti-coalgebra linear automorphism of $U_q(\g)$ determined by the assignment
$e_{\al}\to - f_{\al}$, $q^{h_\al}\to q^{-h_\al}$. Define also an anti-algebra linear involution on $\C_q[\C^N]$ by $\iota (x_i)=(x_{-i})$. They
are compatible with the action $\tr$, that is,
$\iota (u\tr x)= \theta(u)\tr \iota(x)$ for all $u\in U_q(\g)$, $x\in \C_q[\C^N]$.
\begin{lemma}
\label{delta-inv}
  For all $k\in \Z_+$ the monomials $x_0^k$ are killed by $e_{\dt}$ and $f_\dt$.
\end{lemma}
\begin{proof}
Put  $c_k=\frac{q^{-k}-1}{q^{-1}-1}$ for $k\in \Z_+$. Since $f_{\al_2} x^k_0=0$, the equality $f_{\dt}\tr x_0^k=0$  follows from
\be
f_{\al_1}x_0^k&=&
-x_{-1} x_0^{k-1}c_k,
\nn\\
f_{\al_1}^2x_0^k&=&x_{-1}^2 x_0^{k-2}qc_{k-1}c_k,
\nn\\
f_{\al_2}f_{\al_1}^2x_0^k&=&-x_{-2} x_{-1}x_0^{k-2}c_{k-1}c_k[2]_q,
\nn\\
f_{\al_1}f_{\al_2}f_{\al_1}x_0^k
&=&-x_{-2}x_{-1} x_0^{k-2}c_{k-1}c_k.
\nn
\ee
For $e_\dt$, it is obtained by applying the involution $\iota$, as $0=\iota(e_\dt\tr x_0^k)=-f_{\dt}\tr x_0^k=0$.
\end{proof}

The q-version of the quadratic invariant is $C_q=\frac{1}{1+q}x_0^2+\sum_{i=1}^{n}q^{i-1}x_ix_{-i}\in \C_q[\C^N]$.
Let $\Pg_q^m\subset \C_q[\C^N]$ denote the vector space of polynomials of degree $m$, and $\Hg_q^{m}$ the irreducible submodule of  harmonic polynomials of degree $m$.
Then
$$
\C_q[\C^N]=\op_{m=0}^\infty \Pg_q^m, \quad \Pg_q^m=\op_{l=0}^{[\frac{m}{2}]} C^l_q\Hg_q^{m-2l}.
$$
Let $\Pg^m$ and $\Hg^m$ denote their classical counterparts.
\begin{propn}
\label{invariants}
  For any finite dimensional $U_q(\g)$-module $V_q$, $\dim V_q^\k$ is equal to $\dim V^\k$
  of the classical $\k$-invariants.
\end{propn}
\begin{proof}
It is sufficient to show that  $\dim (\Pg^m_q)^\k=\dim (\Pg^m)^\k$.
  In the classical limit, the trivial $\k$-submodule   in $\Hg^{m-2l}$ is multiplicity free, so its dimension
 in $\Pg^m$ is $[\frac{m}{2}]+1$.
On the other hand, the subspace of $U_q(\l)$-invariants is spanned by $\{C^l_qx_0^{m-2l}\}_{l=0}^{[\frac{m}{2}]}$ and has
the same dimension.
Since all $U_q(\l)$-invariants are killed by $e_{\dt}, f_\dt$, in view of  Lemma \ref{delta-inv}, this proves
the statement.
\end{proof}
\appendix
\section{}
\begin{lemma}
\label{lemma_XYZ}
Suppose $x,y,z$ satisfy the relations
$$
[y,[y,x]_q]_{\bar q}=0,\quad [y,[y,z]_q]_{\bar q}=0,\quad [x,z]=0.
$$
Then $[[x,y]_{\bar q},[y,z]_q]=0,\quad[y,[x,[y,z]_q]_{q}]=0.$
\end{lemma}
\begin{proof}
The proof is based on the "Jacobi identity"
$$
[X,[Y,Z]_a]_b=[[X,Y]_c,Z]_{\frac{ab}{c}}+c[Y,[X,Z]_{\frac{b}{c}}]_{\frac{a}{c}}, \quad
$$
which holds true for all elements $X,Y,Z$ of an associative algebra and scalars $a,b,c$ with invertible $c$.
Apply it to the equalities
$$0=[x,[y,[y,z]_q]_{\bar q}]_{\bar q^2}=[z,[y,[y,x]_q]_{\bar q}]_{\bar q^2}=0$$
with $a=\bar q$, $b=\bar q^2$,  $c=\bar q$, and rewrite them    as
$$
0=[[x,y]_{\bar q},[y,z]_q]_{\bar q^2}+\bar q[y,[x,[y,z]_{q}]_{\bar q}]_{}
=
[[z,y]_{\bar q},[y,x]_q]_{\bar q^2}+\bar q[y,[z,[y,x]_q]_{\bar q}]_{}=0.
$$
Observe that  the second terms cancel due to $[x,[y,z]_{q}]_{\bar q}=[[x,y]_{\bar q},z]_{q}=[z,[y,x]_q]_{\bar q}$. Then
$$
[[x,y]_{\bar q},[y,z]_q]_{\bar q^2}
=
[[z,y]_{\bar q},[y,x]_q]_{\bar q^2}=[[y,z]_{q},[x,y]_{\bar q}]_{\bar q^2}.
$$
This yields $(1+q^{-2})[[y,z]_{q},[x,y]_{\bar q}]=0$, which proves the first formula.
Using the "Jacobi identity" with $X=x$, $Y=y$, $Z=[y,z]_q$,  $a=\bar q$, $b=1$, and $c=\bar q$,  we get
$$
0=q[x,[y,[y,z]_q]_{\bar q}]_{}=q[[x,y]_{\bar q},[y,z]_q]+[y,[x,[y,z]_{q}]_{q}]_{}=[y,[x,[y,z]_{q}]_{q}]_{},
$$
which proves the second formula.
\end{proof}

\end{document}